\documentclass[12pt,reqno]{article}

\usepackage{xcolor}
\usepackage{amssymb}
\usepackage{amsmath}
\usepackage{amsthm}
\usepackage{amsfonts} 
\usepackage{amscd}
\usepackage{graphicx}

\usepackage[colorlinks=true,
linkcolor=webgreen,
filecolor=webbrown,
citecolor=webgreen]{hyperref}

\definecolor{webgreen}{rgb}{0,.5,0}
\definecolor{webbrown}{rgb}{.6,0,0}

\usepackage{color}
\usepackage{fullpage}
\usepackage{float}

\usepackage{latexsym}
\usepackage{accents}

\usepackage{graphics}
\numberwithin{equation}{section}

\DeclareMathOperator{\Li}{Li}
\DeclareMathOperator{\Cl}{Cl}
\DeclareMathOperator{\Ein}{Ein}

\begin{document}

\theoremstyle{plain}
\newtheorem{theorem}{Theorem}
\newtheorem{corollary}[theorem]{Corollary}
\newtheorem{lemma}[theorem]{Lemma}
\newtheorem{remark}{Remark}
\newtheorem{example}{Example}

\newcommand{\lrf}[1]{\left\lfloor #1\right\rfloor}

\begin{center}
\vskip 1cm{\LARGE\bf
On a Problem of Mez\H{o} and its Generalizations to Three Classes of Rational Zeta Series \\
\vskip .11in }

\vskip 1cm

{\large

\vskip 0.2 in

Kunle Adegoke \\
Department of Physics and Engineering Physics, \\ Obafemi Awolowo University, Ile-Ife\\ Nigeria \\
\href{mailto:adegoke00@gmail.com}{\tt adegoke00@gmail.com}

\vskip 0.2 in

Robert Frontczak\footnote{Statements and conclusions made in this article by R.~Frontczak are entirely those of the author.
They do not necessarily reflect the views of LBBW.} \\
Landesbank Baden-W\"urttemberg, Stuttgart\\  Germany \\
\href{mailto:robert.frontczak@lbbw.de}{\tt robert.frontczak@lbbw.de}

\vskip 0.2 in

Taras Goy  \\
Faculty of Mathematics and Computer Science\\
Vasyl Stefanyk Precarpathian National University, Ivano-Frankivsk\\ Ukraine\\
\href{mailto:taras.goy@pnu.edu.ua}{\tt taras.goy@pnu.edu.ua}}

\end{center}

\vskip .2 in

\begin{abstract}
We evaluate in closed form three special classes of alternating zeta series with one and two additional parameters. Two classes are expressed as linear combinations of polylogarithms while for the third class we prove an expression involving the incomplete gamma function and the exponential integral. We also present some related series that can be deduced from the main results as well as some series with Fibonacci and Lucas numbers as coefficients. Particular cases of the series presented here will be rediscoveries of identities established by Zhang and Williams, Choi and Srivastava, and Orr, among others. We will also rediscover a series identity published by Mez\H{o} in 2015 as a problem proposal in the American Mathematical Monthly.
\end{abstract}

\section{Motivation}

This paper has two sources of motivation. The first source is a problem proposal by Mez\H{o} from 2015 that appeared in the American Mathematical Monthly \cite{Mezo1}. It asks to prove the identity
\begin{equation}\label{mezo_id}
\frac{1}{2\pi} \Li_2(e^{-2\pi}) = \ln(2\pi) - 1 - \frac{5\pi}{12} - \sum_{k=1}^\infty \frac{(-1)^k \zeta(2k)}{k(2k+1)},
\end{equation}
where
\begin{equation*}
\zeta (s) = \sum_{k=1}^{\infty} \frac{1}{k^{s}}, \qquad \Re(s) >1,
\end{equation*}
is the Riemann zeta function and
\begin{equation*}
\Li_2(z) = \sum_{k=1}^\infty \frac{z^k}{k^2}, \qquad |z|<1,
\end{equation*}
is the dilogarithm. During the course of solving this problem, we found an interesting generalization for which we could provide
two different proofs. Searching deeper in this direction we got familiar with the recent papers by Orr \cite{lupu,orr}
which became the second source of motivation for writing this paper.
Orr has derived beautiful results for two families of rational zeta series which he expressed using
the Clausen functions $\Cl_n(x)$. One such evaluation involving $\zeta(2n)$ is \cite[Eq. (2.4)]{orr}
\begin{equation*}
-2 \sum_{n=0}^\infty \frac{\zeta(2n) z^{2n}}{2n+p} = \sum_{k=0}^p \frac{p! (-1)^{\lfloor (k+3)/2 \rfloor}}{(p-k)! (2\pi z)^k} \Cl_{k+1} (2\pi z)
+ \delta_{\lfloor p/2 \rfloor,p/2} \frac{p! (-1)^{p/2}}{(2\pi z)^p} \zeta(p+1),
\end{equation*}
where $\delta_{j,k}$ is the Kronecker delta function. Orr \cite[Eq. (3.5)]{orr} also evaluates series of the form
\begin{equation*}
\sum_{n=1}^\infty \frac{\zeta(2n) z^{2n}}{(2n)(2n+1)\cdots (2n+m-1)(2n+m+p)}.
\end{equation*}
Here, $\Cl_n (x)$ are the Clausen functions defined by
\begin{equation*}
\Cl_1(x) = - \ln\left (2 \sin\left (\frac{x}{2}\right )\right),\quad |x|<2\pi,
\end{equation*}
and, for $n\geq2$, by \cite[Formulas (7.9) and (7.10)]{lewin81}
\begin{equation}\label{Cl_def}
\Cl_n(x) = \begin{cases}\displaystyle
\Im\big(\Li_n(e^{ix})\big)=\sum_{k=1}^{\infty}\frac{\sin(kx)}{k^n}, & \text{if $n$ is even;} \\\displaystyle
\Re\big(\Li_n(e^{ix})\big)=\sum_{k=1}^{\infty}\frac{\cos(kx)}{k^n}, & \text{if $n$ is odd.}
\end{cases}
\end{equation}
See also \cite{tric} for new information about the Clausen functions. \\

In this article, for a positive integer $n$, we first consider series of the form
\begin{equation*}
P(n,z) = \sum_{k = 1}^\infty \frac{(- 1)^k \zeta (2k) z^{2k}}{k(2k + n)}, \qquad 0<|z|\leq 1.
\end{equation*}
In addition, for positive integers $m$ and $n$ with $m\neq n$, we also treat the class of rational series
\begin{equation*}
P(m,n,z) = \sum_{k = 1}^\infty \frac{(- 1)^k \zeta (2k) z^{2k}}{k(2k+m)(2k + n)}, \qquad 0<|z|\leq 1,
\end{equation*}
and its degenerated counterpart for $m\geq1$
\begin{equation*}
Q(m,z) = \sum_{k = 1}^\infty \frac{(- 1)^k \zeta (2k) z^{2k}}{k(2k+m)^2}, \qquad 0<z\leq 1.
\end{equation*}

We express $P(n,z)$ and $P(m,n,z)$ as linear combinations of polylogarithms $\Li_s (x)$.
For $Q(m,z)$ we prove an identity involving the incomplete gamma function $\Gamma(a,x)$ and the exponential integral $\Ein(x)$. We also present some related series that can be deduced from the main results as well as some series with Fibonacci and Lucas numbers as coefficients.
The paper concludes with a discussion of further possible generalizations of the present results to series with an arbitrary power $k^p,p\geq 1,$ in the denominator of $P(n,z),P(m,n,z)$ and $Q(m,z)$, respectively.

\section{A generalization of the zeta identity of Mez\H{o}}

First, we prove the following generalization of \eqref{mezo_id}, for which we offer two proofs.
\begin{theorem}\label{thm1}
	Let $n$ be a positive integer. For all $z$ such that $0<|z|\le 1$, we have the identity
	\begin{align}\label{main_id1}
	P(n,z)&=\sum_{k = 1}^\infty  \frac{{( - 1)^{k} \zeta (2k)z^{2k} }}{{k(2k + n)}}\nonumber\\
	&= -\frac{1}{n^2} - \frac{\pi z}{n + 1} + \frac{\ln (2\pi z)}{n} + \frac{(n - 1)!}{(2\pi z)^n}{\zeta (n + 1)}
	- (n - 1)!\sum_{j = 1}^{n} \frac{\Li_{ j+1}( e^{- 2\pi z})}{(n-j)!(2\pi z)^j},
	\end{align}
	where $\displaystyle \Li_s(z)=\sum_{k=1}^\infty \frac{z^k}{k^s}$ is the polylogarithm of order $s$ evaluated at $z$.
\end{theorem}
\noindent{\it First proof.} Let
\begin{equation}\label{sum_s1}
S_1(z) = \sum_{k=1}^\infty \frac{(-1)^k \zeta(2k)}{k} z^{2k}.
\end{equation}
For a variable $y$, we have
\[
y^{n - 1} S_1 (zy ) = \sum_{k = 1}^\infty \frac{{( - 1)^k \zeta (2k) }}{k}z^{2k} y^{2k + n - 1}
\]
and hence
\begin{equation}\label{eq.u4gd1cu}
P(n,z) = \int_0^1 {y^{n - 1} S_1 (zy )dy}.
\end{equation}

The evaluation of $S_1(z)$ is a classical result and equals
\begin{align*}
S_1(z) &= - \ln\Big (\frac{\sinh(\pi z)}{\pi z}\Big ) = \pi z + \ln\Big (\frac{2\pi z}{e^{2\pi z}-1}\Big ) \\
&= - \pi z + \ln (2\pi z) - \ln \left( {1 - e^{ - 2\pi z} } \right);
\end{align*}
so that
\begin{equation}\label{eq.hyjpkxw}
S_1 (zy ) = - \pi zy + \ln (2\pi z) + \ln y - \ln \left( {1 - e^{ - 2\pi zy } } \right).
\end{equation}
Thus, from \eqref{eq.u4gd1cu} and \eqref{eq.hyjpkxw} we have
\begin{align}\label{eq,n2n2u60}
P(n,z) &=  - \pi z\int_0^1 {y^n dy}  + {\ln (2\pi z)}\int_0^1 {y^{n - 1} dy}  + \int_0^1 {y^{n - 1} \ln y}\,dy \nonumber\\
&\quad \,- \int_0^1 y^{n - 1} \ln\!\big( {1 - e^{ - 2\pi yz } } \big)dy \nonumber\\
&= - \frac{\pi z}{{n + 1}} + \frac{{\ln (2\pi z)}}n - \frac1{n^2} - \int_0^1 {y^{n - 1} \ln\!\big( {1 - e^{ - 2\pi yz } } \big)dy}.
\end{align}
It now remains to evaluate the remaining integral in \eqref{eq,n2n2u60}. Let
\[
I(z) = - \int_0^1 {y^{n - 1} \ln\!\big( {1 - e^{ - 2\pi yz } } \big)dy} .
\]
Using the identity
\[
\ln \Big(1 - \frac{1}{x}\Big) = - \sum_{m = 1}^\infty  \frac{1}{mx^m},\qquad |x|>1,
\]
we have
\[
I(z) = \int_0^1 y^{n - 1} \sum_{m = 1}^\infty \frac{e^{- 2\pi zy m} }{m} dy = \sum_{m = 1}^\infty \frac{1}{m}\int_0^1 {y^{n - 1} e^{- 2\pi zy m} dy},
\]
where the interchange of integration and summation is justified by uniform convergence.

Next, using the standard integral \cite[Entry 3.351]{GrRyz}
\begin{equation}\label{GrRyz}
\int_0^1 x^w e^{-\mu x} dx = \frac{w!}{\mu^{w+1}} - e^{-\mu}\frac{w!}{\mu^{w+1}}\sum_{j=0}^w \frac{\mu^j}{j!},
\end{equation}
we have
\[
I'(z) = \int_0^1 {y^{n - 1} e^{ - 2\pi zy m} dy} = \frac{{(n - 1)!}}{{(2\pi zm)^{n} }} - \frac{(n - 1)!e^{ - 2\pi zm}}{(2\pi zm)^n}  \sum_{j = 0}^{n - 1} \frac{(2\pi zm)^{j}}{j!}
\]
and hence
\begin{equation}\label{eq.egazo94}
I(z) = \frac{{(n - 1)!}}{{(2\pi z)^{n} }}\sum_{m = 1}^\infty  {\frac{1}{{m^{n + 1} }}}  - \frac{(n - 1)!}{(2\pi z)^n} \sum_{j = 0}^{n - 1} {\frac{(2\pi z)^{j}}{j!}\sum_{m = 1}^\infty  {\frac{e^{ - 2\pi m z}}{m^{n - j + 1}}} }.
\end{equation}

Plugging \eqref{eq.egazo94} into \eqref{eq,n2n2u60} gives the identity stated in the theorem.

\medskip
\noindent {\it Second proof}. We start with the obvious observation that
\begin{equation*}
\frac{1}{k(2k+n)} = \frac{1}{n}\Big (\frac{1}{k} - \frac{2}{2k+n}\Big ),
\end{equation*}
and hence $nP(n,z) = S_1(z) - 2 S_2(z)$,
where $S_1(z)$ is defined and evaluated as in the first proof and $S_2(z)$ equals
\begin{equation*}
S_2(z) = \sum_{k = 1}^\infty \frac{(- 1)^k \zeta (2k)}{2k + n}z^{2k}.
\end{equation*}
We get
\begin{align*}
S_2(z) & =\sum_{k=1}^\infty \sum_{s=1}^\infty \frac{(-1)^k }{2k+n} s^{-2k} z^{2k} \\
& = \sum_{k=1}^\infty \sum_{s=1}^\infty (-1)^k \Big (\frac{z^2}{s^2}\Big )^k \int_0^1 x^{2k+n-1} dx \\
& = \int_0^1 x^{n-1} \sum_{s=1}^\infty \left (\sum_{k=0}^\infty \Big (-\frac{z^2 x^2}{s^2}\Big )^k - 1 \right ) dx \\
& = -\int_0^1 x^{n+1} z^2 \sum_{s=1}^\infty \frac{dx}{s^2 + z^2x^2}.
\end{align*}

In view of the known identity \cite[Entry 1.421]{GrRyz}
\begin{equation*}
\coth(\pi x)=\frac{1}{\pi x}+\frac{2x}{\pi}\sum_{k=1}^\infty \frac{1}{k^2+x^2},
\end{equation*}
we get
\begin{equation*}
\sum_{s=1}^\infty \frac{1}{s^2 + z^2x^2} = \frac{1}{2z^2 x^2} \big ( \pi z x \coth(\pi z x ) - 1 \big ),
\end{equation*}
and hence
\begin{align*}
S_2(z) & =  \frac{1}{2n} - \frac{\pi z}{2} \int_0^1 x^{n} \coth(\pi z x)dx \\
& =  \frac{1}{2n} - \frac{\pi z}{2} \int_0^1 x^{n} \Big (1 + \frac{2e^{-2\pi zx}}{1-e^{-2\pi zx}}\Big ) dx \\
& =  \frac{1}{2n} - \frac{\pi z}{2} \Big ( \frac{1}{n+1} + 2 \sum_{m=1}^\infty \int_0^1 x^{n} e^{-2\pi z m x}dx \Big ).
\end{align*}

Again we can use \eqref{GrRyz} to simplify. The result is
\begin{align*}
P(n,z) & =  - \frac{\pi z}{n} + \frac{\ln (2\pi z)}{n} - \frac{\ln (1-e^{-2\pi z})}{n} - \frac{1}{n^2} + \frac{\pi z}{n(n+1)} \\
& \quad\, + \frac{(n-1)!}{(2\pi z)^n} \Big ( \zeta(n+1) - \sum_{j=0}^{n} \frac{(2\pi z)^j}{j!} \Li_{n+1-j} (e^{-2\pi z}) \Big ).
\end{align*}
As $\Li_1(z)=-\ln (1-z)$ the proof is completed. \hfill{$\square$}
\begin{remark}
	Identity \eqref{mezo_id} is deduced from the evaluation of $P(1,1)$.
\end{remark}
\begin{example} Taking particular values of the $n$ and $z$ in \eqref{main_id1} leads to the following series:
	\begin{align*}
	\sum_{k=1}^\infty  \frac{(-1)^{k-1}\zeta(2k)}{k(2k+1)}  &= 1+\frac{5\pi}{12}  - \ln(2\pi)+\frac{1}{2\pi}\Li_2(e^{-2\pi}),\\
	\sum_{k=1}^\infty  \frac{(-1)^{k-1}\zeta(2k)}{2^k k(2k+1)} & = 1+\frac{\pi}{3\sqrt2}  - \ln(\sqrt2\pi)+\frac{1}{\sqrt2\pi}\Li_2(e^{-\sqrt2\pi}),\\
	\sum_{k=1}^\infty  \frac{(-1)^{k-1}\zeta(2k)}{4^k k(2k+1)} & = 1+\frac{\pi}{12}  - \ln\pi+\frac{1}{\pi}\Li_2(e^{-\pi}),\\
	\sum_{k=1}^\infty  \frac{(-1)^{k-1}\zeta(2k)}{k(k+1)}  &= \frac12 +\frac{2\pi}{3}  - \ln(2\pi)-\frac{\zeta(3)-\Li_3(e^{-2\pi})-2\pi\Li_2(e^{-2\pi})}{2\pi^2}.
	\end{align*}
\end{example}
\begin{theorem}
	Let $n$ be a positive integer and let $z$ be a real number such that $0<|z|\leq 1$. Then
	\begin{align}\label{eq.mag6ovt}
	\sum_{k = 1}^\infty& \frac{\zeta (2k)z^{2k}}{k(k + n)} = -\, \frac{1}{2n^2}+\frac{\ln (2\pi z)}{n} + \frac{( - 1)^n 2(2n-1)! }{(2\pi z )^{2n} }\, \zeta (2n + 1)\nonumber\\
	&\qquad\qquad\quad -2(2n-1)! \sum_{j = 1}^{n} (- 1)^j
	\frac{(2n+1-2j)\Cl_{ 2j+1} (2\pi z) + 2\pi z \Cl_{2j} (2\pi z)}{(2n+1-2j)!(2\pi z)^{2j} },\\
	\label{eq.no0fbhh}
	\sum_{k = 1}^\infty &
	\frac{\zeta (2k)z^{2k}}{ k(2k - 1 + 2n)} =  - \frac{1}{(2n - 1)^2} + \frac{\ln \big(\pi z\csc(\pi z)\big)}{2n -1}\nonumber\\
	&\qquad\qquad\quad\quad\,\,\, -(2n-2)!\sum_{j =1}^{n} (- 1)^j \frac{(2n+1-2j)\Cl_{2j}(2\pi z)-2\pi z\Cl_{2j-1}(2\pi z)}{(2n+1-2j)!(2\pi z)^{2j-1}},
	\end{align}
	where $\Cl_j(x)$ are the Clausen functions defined in \eqref{Cl_def}.
\end{theorem}
\begin{proof}
	Evaluate $P(2n,-iz)$ and $P(2n - 1,-iz)$, where $i$ denotes the imaginary unit, and simplify.
\end{proof}

The following numerical relations of the Clausen functions and polylogarithm are known \cite[Sections 4.3, 4.5, 7.2, 7.3, 7.5]{lewin81}.
\begin{lemma} We have
	\begin{gather*}
	\Cl_2(n\pi) = 0,\,\,\, n\in \mathbb{Z},\qquad 
	\Cl_2\Big(\frac{\pi}{2}\Big) =G, \qquad \Cl_2\Big(\frac{3\pi}{2}\Big)= -\,G,\\
	\Cl_2\Big(\frac{\pi}{3}\Big) = \frac32\Cl_2\Big(\frac{2\pi}{3}\Big), \qquad \Cl_2\Big(\frac{\pi}{6}\Big) + \Cl_2\Big(\frac{5\pi}{6}\Big) = \frac{4G}3,\\
	\Cl_{2n + 1} (\pi ) = (2^{- 2n} - 1  )\zeta (2n + 1),\qquad \Cl_{2n + 1} (2\pi )= \zeta (2n + 1),\\
	\Cl_{2n + 1} \Big(\frac{\pi}{2}\Big) = 2^{- (2n + 1)} (2^{- 2n} -1 )\zeta (2n + 1),\\
	\Cl_{2n + 1} \Big(\frac{\pi}{3}\Big) = \frac{1}{2}(2^{ - 2n} - 1)(3^{- 2n} -1)\zeta (2n + 1),\\
	\Cl_{2n + 1} \Big(\frac{2\pi}{3}\Big) = \frac{1}{2}(3^{ - 2n} -1 )\zeta (2n + 1),\\
	\Li_n (1) = \zeta (n),\qquad 	\Li_n (- 1) = (2^{1 - n}-1 )\zeta (n),
	\end{gather*}
	where $  G=\sum\limits_{n=1}^{\infty} \frac{(-1)^n}{(2n+1)^2}$ is Catalan's constant.
\end{lemma}
\begin{corollary}
	If $z$ is a real number such that $0<|z|\le1$, then
	\begin{align*}
	\sum_{k = 1}^\infty \frac{\zeta (2k)z^{2k}}{k(k + 1)} &= - \frac{1}{2}  + \ln ( 2\pi z) - \frac{ \zeta (3) - \Cl_3 ( 2\pi z) + \frac{1}{\pi z}\Cl_2 (2\pi z)}{2\pi ^2z^2},\\
	\sum_{k = 1}^\infty \frac{\zeta (2k) z^{2k}}{ k(2k + 1)} &= -1 +\ln (2\pi z)
	+ \frac{\Cl_2 (2\pi z)}{2\pi z }.
	\end{align*}
\end{corollary}
\begin{proof}
	Set $n=1$ in \eqref{eq.mag6ovt} and \eqref{eq.no0fbhh}, respectively.
\end{proof}
\begin{corollary}
	If $n$ is a positive integer, then
	\begin{align}\label{Cor4_1}
	 \sum_{k = 1}^\infty &\frac{\zeta (2k)}{k(k + n)} = - \frac{1}{2n^2}+ \frac{{\ln (2\pi )}}{n}  -2(2n-1)!\sum_{j = 1}^{n - 1} \frac{( - 1)^j \zeta (2j +1)}{(2\pi)^{2j}(2n-2j)!},\\
	\label{Cor4_2}
	\sum_{k = 1}^\infty &\frac{\zeta (2k) }{{k(2k + 2n - 1)}} = - \frac{1}{(2n - 1)^2 } + \frac{\ln (2\pi )}{2n - 1}  - (2n - 2)! \sum_{j = 1}^{n - 1} \frac{( - 1)^{j}\zeta(2j+1)}{(2\pi )^{2j}(2n-2j-1)!}.
	\end{align}
\end{corollary}
\begin{proof}
	Evaluate $2P(2n,i)$ and $P(2n - 1,i)$ using  $$\ln(2\pi i)= \ln(2\pi) + {i\pi}/{2},\qquad \Li_n(e^{-2\pi i})=\zeta(n).$$
\end{proof}
\begin{example} From \eqref{Cor4_1} and
	\eqref{Cor4_2} one gets the following series:
	\begin{align}
	\sum_{k = 1}^\infty {\frac{{\zeta (2k)}}{{k(k + 1)}}} &= - \frac{1}{2} + \ln (2\pi),\nonumber\\
	\sum_{k = 1}^\infty {\frac{\zeta (2k)}{k(k + 2)}} &= - \frac{1}{8} + \frac{\ln (2\pi )}{2} + \frac{3\,\zeta (3)}{2\pi ^2 },\nonumber
	\end{align}
		\begin{align}
	\label{eq.b6s2gol}
	\sum_{k = 1}^\infty {\frac{{\zeta (2k)}}{{k(2k + 1)}}}& = -1 + \ln (2\pi ),\\
	\sum_{k = 1}^\infty {\frac{{\zeta (2k)}}{{k(2k + 3)}}} &=- \frac{1}{9} +  \frac{\ln (2\pi )}{3}  + \frac{\zeta (3)}{2\pi ^2 }.\nonumber
	\end{align}
\end{example}

Identity \eqref{eq.b6s2gol} was also reported by Yun-Fei \cite[Identity (2.54)]{yunfei11}.
\begin{corollary}
	If $n$ is a positive integer, then
	\begin{align} \label{eq.ya1fj3a}
	\sum_{k = 1}^\infty \frac{{\zeta (2k)}}{4^{k}k(k + n)} &= \frac{\ln\pi }{n} - \frac{1}{2n^2}+ \frac{( - 1)^n(2n)!}{\pi^{2n}n }\zeta (2n + 1)\nonumber \\
	&\quad + 2(2n-1)!\sum_{j = 1}^{n - 1} \frac{( - 1)^j (2^{2j}-1)}{(2\pi)^{2j} (2n-2j)!} \zeta (2j+1),\\
	\label{eq.hhcgjow}
	\sum_{k = 1}^\infty \frac{{\zeta (2k) }}{4^{k}k(2k + 2n - 1)} &= \frac{{\ln\pi}}{{2n - 1}} - \frac{1}{{(2n - 1)^2 }}\nonumber \\
	&\quad -(2n - 2)! \sum_{j = 1}^{n - 1} \frac{(-1)^j (2^{2j}-1)} {(2\pi)^{2j}(2n+1-2j)!}\zeta (2j + 1).
	\end{align}
\end{corollary}
\begin{proof}
	Evaluate $2P(2n,i/2)$ and $P(2n - 1,i/2)$.
\end{proof}
\begin{example} At $n=1$ and $n=2$, from \eqref{eq.ya1fj3a}, \eqref{eq.hhcgjow} we obtain
	\begin{align}\label{ex3_1}
	\sum_{k = 1}^\infty {\frac{{\zeta (2k)}}{4^{k} k(k + 1)}} &= - \frac{1}{2} + \ln \pi - \frac{7\zeta (3)}{2\pi ^2 } ,\\
	\sum_{k = 1}^\infty {\frac{{\zeta (2k)}}{4^{k} k(k + 2)}} &= - \frac{1}{8} + \frac{\ln \pi }{2} - \frac{9\,\zeta (3)}{2\pi ^2}  + \frac{93\,\zeta (5)}{4\pi ^4 } ,\nonumber \\
	\label{eq.ycu3i6g}
	\sum_{k = 1}^\infty {\frac{{\zeta (2k)}}{4^{k} k(2k + 1)}} &= \ln \pi - 1,\\
	\label{ex3_4}
	\sum_{k = 1}^\infty {\frac{{\zeta (2k)}}{4^{k} k(2k + 3)}} &= - \frac{1}{9} + \frac{\ln \pi }{3}- \frac{3 \zeta (3)}{2\pi ^2}.
	\end{align}
\end{example}
Identities \eqref{ex3_1}, \eqref{eq.ycu3i6g} and \eqref{ex3_4} were derived by Zhang and Williams \cite[p.~1585]{zhang1}; see also  \cite[Formulas (2.16),  (2.17)]{Dabr} and \cite{Tyler}.
\begin{theorem}
	For all $z$ such that $0<|z|\le 1$,
	\begin{align}\label{eq.og6kibc}
	\sum_{k = 1}^\infty \frac{{\zeta (2k)z^{2k} }}{{k(2k + 1)}} &= -1 +\ln (2\pi |z|)  + \frac{\Cl_2 (2\pi z)}{{2\pi z}},\\
	\label{th6_2}
	\sum_{k = 1}^\infty \frac{{\zeta (2k)z^{2k} }}{{k(k + 1)}} &= - \frac{1}{2}+ \ln (2\pi |z|)  - \frac{{\zeta (3)}}{{2\pi ^2 z^2 }}
	+ \frac{\Cl_3 (2\pi z)}{{2\pi ^2 z^2 }} + \frac{\Cl_2 (2\pi z)}{{\pi z}}.
	\end{align}
\end{theorem}
\begin{proof}
	Evaluate $P(1,-iz)$ and $2P(2,-iz)$, respectively.
\end{proof}
Identity \eqref{eq.og6kibc} is equivalent to, but much simpler and useful than that from \cite[Formula~(566)]{SriChoi12}.
\begin{example} From \eqref{eq.og6kibc} and \eqref{th6_2} we have
	\begin{align}
	\label{Will2}\sum_{k = 1}^\infty \frac{\zeta (2k)}{16^{k} k(2k + 1)} &= \ln \Big( {\frac{\pi }{2}} \Big) - 1 + \frac{{2G}}{\pi },\\
	\label{2.23}\sum_{k = 1}^\infty  \Bigl(\frac{9}{16}\Bigr)^{k}\frac{\zeta (2k)}{k(2k + 1)}  &= \ln \Big( {\frac{{3\pi }}{2}} \Big) - 1 - \frac{{2G}}{{3\pi }},\\
	\label{2.24}\sum_{k = 1}^\infty \frac{{\zeta (2k)}}{{16^k k(k + 1)}} &= \ln \left( {\frac{\pi }{2}} \right) - \frac{1}{2} - \frac{35\,\zeta (3)}{4\pi ^2}  + \frac{{4G}}{\pi },
	\end{align}
	as well as
	\begin{align*}
	\sum_{k = 1}^\infty \frac{{\zeta (2k)}}{9^{k} k(2k + 1)} &= \ln \Big( {\frac{2\pi }{3}} \Big) - 1 - \frac{\sqrt{3}}{9} \pi
	+ \frac{\sqrt{3}}{6\pi }\,	\psi'\Big( \frac{1}{3} \Big),\\
	\sum_{k = 1}^\infty \frac{{\zeta (2k)}}{36^{k} k(2k + 1)} &= \ln \left( {\frac{\pi }{3}} \right) - 1 - \frac{\pi}{\sqrt{3}}
	+ \frac{\sqrt{3}}{2\pi }\,\psi'\Big( \frac{1}{3} \Big),\\
	\sum_{k = 1}^\infty \frac{{\zeta (2k)}}{{64^{k} k(2k + 1)}} &= \ln \left( {\frac{\pi }{4}} \right) - 1 - \frac{\sqrt{2}+1}{4} \pi
	- \frac{(2\sqrt{2}-1)G}{\pi} + \frac{\sqrt{2}}{8\pi }\,\psi' \Big( \frac{1}{8} \Big),
	\end{align*}
	where $ \psi(z) = \frac{\Gamma'(z)}{\Gamma(z)}$ is the digamma function with
	$\psi'(z) = \sum\limits_{n=0}^\infty \frac{1}{(z+n)^2}$.
		Note that in the last three examples we have used evaluations
	\begin{gather*}	
	\Cl_2 \Big(\frac{2\pi}{3}\Big)=\frac{\sqrt3}{9}\Big(\psi'\Big(\frac13\Big)-\frac{2\pi^2}{3}\Big), \qquad \Cl_2 \Big(\frac{\pi}{3}\Big)=\frac{\sqrt3}{6}\Big(\psi'\Big(\frac13\Big)-\frac{2\pi^2}{3}\Big),\\ 
	\Cl_2 \Big(\frac{\pi}{4}\Big)=\frac{1}{32}\Big(\sqrt2\psi'\Big(\frac18\Big)-2(\sqrt2+1)\pi^2-8(2\sqrt2-1)G\Big),
	\end{gather*}
	which were derived by Grosjean  \cite{Grosjean}.
\end{example}

Summation formulas \eqref{Will2}, \eqref{2.23}, \eqref{2.24} are known results. For example, \eqref{Will2} appears in \cite[Entry (54.5.6)]{hansen}.   See also \cite[Formulas (698),~(699)]{SriChoi12} and \cite[Formulas (5.21),  (5.22)]{ChoiSri}.
\begin{theorem}
	For all $z$ such that $0<|\ln z|\le 2\pi$,
	\begin{equation}\label{eq.juxtpm5}
	\sum_{k = 1}^\infty \frac{{ (\ln z)^{2k + 1} B_{2k}}}{k(2k + 1)!}  =   \frac{{\pi ^2 }}{3} - \frac{1}{2}\ln ^2 z + 2\ln z\,\big(1- \ln ( - \ln z) \big) - 2\Li_2 (z),
	\end{equation}
	where $B_j$ denotes the Bernoulli numbers.
\end{theorem}
\begin{proof}
	Evaluate $P(1,-\frac{\ln z}{2\pi})$ and use
	$\zeta (2n) = \frac{( - 1)^{n + 1}(2\pi )^{2n} }{2(2n)!}B_{2n}, \,\,n\geq1.
	$
\end{proof}

It is instructive to compare \eqref{eq.juxtpm5} with  \cite[Identity (1.76)]{lewin81}:
\[
\Li_2 (e^{ - z} ) = \frac{{\pi ^2 }}{6} + z\ln z - z - \frac{{z^2 }}{4} + \frac{{B_1 z^3 }}{{2 \cdot 3\cdot2!}} - \frac{{B_2 z^5 }}{{4 \cdot 5\cdot4!}} +  \cdots.
\]
\begin{corollary}
	If $z$ is a real number such that $0<| z|\leq2\pi$, then
	\[
	\sum_{k = 1}^\infty  \frac{B_{2k}z^{2k + 1}}{k(2k + 1)!}  =  \frac{4z+z^2}{2} - 2z\ln z - \frac{\pi ^2}{3} + 2\Li_2 (e^{ - z}).
	\]
\end{corollary}
\begin{proof}
	Evaluate \eqref{eq.juxtpm5} at $z=e^{-z}$ with $z>0$.
\end{proof}

In particular,
\[
\sum_{k = 1}^\infty  \frac{{B_{2k} }}{k(2k + 1)!}  = \frac{5}{2} - \frac{{\pi ^2 }}{3} + 2\Li_2 (e^{ - 1} ).
\]
\begin{corollary}
	For all $z$ such that $0<|z|\le 2\pi$,
	\begin{equation}\label{eq.nbj89nq}
	\sum_{k = 1}^\infty  {\frac{{B_{2k} z^{2k} }}{k(2k)!}}  = 2\ln \left( \frac{2}{z}\sinh\left(\frac{z}{2}\right)\right).
	\end{equation}
\end{corollary}
\begin{proof}
	Differentiate \eqref{eq.juxtpm5} with respect to $z$ and write $e^{-z}$ for $z$.
\end{proof}

Differentiating \eqref{eq.nbj89nq} with respect to $z$, in the next corollary we obtain the generating function of even indexed Bernoulli numbers.
\begin{corollary}
	For all real $z$ such that $0<|z|\le 2\pi$,
	\[
	\sum_{k = 0}^\infty  \frac{B_{2k}z^{2k} }{(2k)!}  = \frac{z}{2} + \frac{z}{{e^z  - 1}}.
	\]
\end{corollary}
\begin{corollary}
	For all $z$ such that $0<|z|\le 1$,
	\begin{equation}\label{eq.y8rpuz3}
	\sum_{k = 1}^\infty \frac{(- 1)^{k-1} \zeta (2k) z^{2k}}{2k + 1} = -\,\frac{1}{2} + \frac{\pi z}{4}  - \frac{\pi}{24z}
	+ \frac{1}{2}\ln\big(\! -2\sinh(\pi z)\big) + \frac{\Li_2 (e^{2\pi z} )}{4\pi z}  .
	\end{equation}
\end{corollary}
\begin{proof}
	Differentiate $P(1,z)$ with respect to $z$.
\end{proof}
\begin{corollary}
	For all real $z$ such that $0<|z|<1$,
	\begin{equation}\label{Cor12}
	\sum_{k = 1}^\infty  {\frac{{\zeta (2k)z^{2k} }}{{2k + 1}}}  = \frac{1}{2} - \frac{1}{2}\ln\big(2\sin(\pi z)\big) - \frac{\Cl_2(2\pi z)}{4\pi z}.
	\end{equation}
\end{corollary}
\begin{proof}
	Write $iz$ for $z$ in \eqref{eq.y8rpuz3} and take the imaginary part.
\end{proof}

Note that Formula \eqref{Cor12} is given in a slightly different way in \cite[Formula (490)]{SriChoi12}. Also, one can find it in \cite[Formula (54.5.4)]{hansen} and \cite[Formula (2.18)]{wu}.
\begin{example} Taking $z=1/2$, $z=1/4$ and $z=3/4$ in \eqref{Cor12} yield the following identities:
	\begin{align}\label{eq.bruala5}
	\sum_{k = 1}^\infty  {\frac{{\zeta (2k)}}{{4^{k} (2k + 1)}}}  &= \frac{1}{2} - \frac{\ln 2}{2},\\
	\label{2.32}\sum_{k = 1}^\infty  {\frac{{\zeta (2k)}}{{16^{k} (2k + 1)}}}  &= \frac{1}{2} - \frac{\ln 2}{4} - \frac{G}{\pi },\\
	\label{2.33}\sum_{k = 1}^\infty \Bigl(\frac{9}{16}\Bigr)^{k}\frac{\zeta (2k)}{2k + 1}&  = \frac{1}{2} - \frac{\ln 2}{4} +  \frac{G}{3\pi}.
	\end{align}
\end{example}

Identity \eqref{eq.bruala5} is found in \cite[p.~313, Formula (493)]{SriChoi12} and \cite{orr}.  One can find  \eqref{2.32} and \eqref{2.33}  in \cite[Formulas (670) and (671)]{SriChoi12}.
\begin{theorem}
	For all $z$ such that $0<|z|\le 1$,
	\begin{equation}\label{eq.i63cmw9}
	\sum_{k = 1}^\infty  {\frac{{( - 1)^{k-1} \zeta (2k)z^{2k} }}{{(2k + 1)(k + 1)}}}  = -\frac12 - \frac{{\pi z }}{6} -  \frac{\pi}{12z} - \frac{{\zeta (3)}}{2\pi ^2z^2 }  + \frac{{\Li_3 (e^{2\pi z} )}}{{2\pi ^2z^2 }} - \frac{{\Li_2 (e^{2\pi z} )}}{2\pi z}.
	\end{equation}
\end{theorem}
\begin{proof}
	Multiply through \eqref{eq.y8rpuz3} by $z$ and integrate with respect to $z$.
\end{proof}
\begin{corollary}
	For all real $z$ such that $0<|z|\le 1$,
	\begin{equation}\label{Cor14}
	\sum_{k = 1}^\infty  {\frac{{\zeta (2k)z^{2k} }}{{(k + 1)(2k + 1)}}}  = \frac{1}{2}  - \frac{{\zeta (3)}}{2\pi ^2 z^2 } + \frac{{\Cl_3 (2\pi z)}}{2\pi ^2 z^2 } + \frac{{\Cl_2 (2\pi z)}}{2\pi z}.
	\end{equation}
\end{corollary}
\begin{proof}
	Write $iz$ for $z$ in \eqref{eq.i63cmw9} and take real parts.
\end{proof}
\begin{example} For certain $z$, from \eqref{Cor14} we have
	\begin{align}
	\sum_{k = 1}^\infty  \frac{{\zeta (2k)}}{{(k + 1)(2k + 1)}}  &= \frac12,\nonumber\\
	\label{Euler}
	\sum_{k = 1}^\infty  \frac{{\zeta (2k)}}{{4^{k} (k + 1)(2k + 1)}}  &=  \frac12 - \frac{7\zeta (3)}{2\pi ^2 } ,\\
	\label{16^k}\sum_{k = 1}^\infty  {\frac{{\zeta (2k)}}{16^{k} (k + 1)(2k + 1)}} & =  \frac12 - \frac{35\zeta (3)}{4\pi ^2 } + \frac{2G}{\pi },\\
	\sum_{k = 1}^\infty  \Bigl(\frac{9}{16}\Bigr)^k \frac{\zeta (2k)}{(k+1)(2k + 1)} & =  \frac{1}{2} - \frac{35\zeta (3)}{36\pi ^2} -  \frac{2G}{3\pi},\nonumber\\
	\sum_{k = 1}^\infty  \frac{\zeta (2k)}{36^k(k+1)(2k + 1)} & =  \frac{1}{2} -\frac{\pi}{\sqrt3}- \frac{12\zeta (3)}{\pi ^2} + \frac{\sqrt3}{2\pi} \,\psi'\Big(\frac13\Big). \nonumber
	\end{align}
\end{example}

The series representation \eqref{Euler} is contained in one of Euler's papers and was rediscovered by many mathematicians (see \cite{Ewell} and \cite{zhang1} among others). Identity \eqref{16^k} is also known (\cite[Formula (5.10)]{Adam}, \cite[Formula (16)]{lupu}).

\section{Connection with second order sequences}

In this section we study some series involving the  Riemann zeta function and Fibonacci (Lucas) numbers.
The results are closely related to our studies in  \cite{KA-NNTDM,Fro_NNTDM,MatStud}.

As usual, let $F_n$ and $L_n$ denote the $n$-th Fibonacci and Lucas numbers, both satisfying the recurrence $w_n = w_{n-1} + w_{n-2}$ for $n \geq 2$, but with the initial values $F_0 = 0$, $F_1 = 1$ and $L_0 = 2$, $L_1 = 1$, respectively. The Binet formulas are
\begin{equation}\label{bine}
F_n = \frac{\alpha^n-\beta^n}{\alpha-\beta}, \quad L_n = \alpha^n + \beta^n,\quad n\geq0,
\end{equation}
where $\alpha = {(1+\sqrt{5})}/{2}$ and $\beta=-{1}/{\alpha}={(1-\sqrt{5})}/{2}$.
\begin{theorem}
	If $n$ is a positive integer and $z$ is any real number such that 	$0<|z|\leq1$, then
	\begin{align}\label{th15_F}
	\sqrt5\sum_{k = 1}^\infty   \frac{(-1)^{k-1}F_{2k}\zeta (2k)z^{2k}} {k(2k+n)} &= \frac{\pi z}{n+1}-\frac{2\ln\alpha}{n}+\frac{2(n-1)!\sinh(n\ln\alpha)}{(2\pi z)^n}\zeta(n+1)\nonumber\\
	& -(n-1)!\sum_{j=1}^{n}\frac{\alpha^{j}\Li_{j+1}\!\big(e^{2\pi\beta z}\big)-(-\beta)^{j}\Li_{j+1}\!\big(e^{-2\pi\alpha z}\big)}{(n-j)!(2\pi z)^{j}},\\
	\label{th15_L}
	\sum_{k = 1}^\infty  \frac{(-1)^{k-1}L_{2k}\zeta (2k)z^{2k}} {k(2k+n)}&= \frac{\sqrt5\pi z}{n+1} -\frac{2\ln(2\pi z)}{n}-\frac{2(n-1)!\cosh(n\ln\alpha)}{(2\pi z)^n}\zeta(n+1)\nonumber\\
	&\hspace{-1cm}+\frac{2}{n^2} +(n-1)!\sum_{j=1}^{n}\frac{\alpha^{j}\Li_{j+1}\!\big(e^{2\pi\beta z}\big)+(-\beta)^{j}\Li_{j+1}\!\big(e^{-2\pi\alpha z}\big)}{(n-j)!(2\pi z)^{j}}.
	\end{align}
\end{theorem}
\begin{proof}
	Evaluate  $P(n,\alpha z)$ and $P(n,-\beta z)$  and combine these equations according to \eqref{main_id1} and the Binet formulas \eqref{bine}.
\end{proof}
\begin{theorem}
	If $z$ is any real number such that $0<|z|\leq {1}/{\alpha}$, then
	\begin{align}
	\label{th16_F}
		\sum_{k = 1}^\infty \frac{F_{2k}\zeta (2k)z^{2k}} {k(2k+1)} &= \frac{2\ln\alpha}{\sqrt5} -\frac{\beta\Cl_2(2\pi\alpha z)- \alpha\Cl_2(2\pi\beta z)}{2\sqrt5\pi z},\\
	\label{th16_L}
	\sum_{k = 1}^\infty \frac{L_{2k}\zeta (2k)z^{2k}} { k(2k+1)} & = -2+ 2\ln(2\pi z)-\frac{\beta\Cl_2(2\pi\alpha z)+\alpha\Cl_2(2\pi\beta z)}{2\pi z }.
	\end{align}
\end{theorem}
\begin{proof}
	Evaluate  $P(1,i\alpha z)$ and $P(1,-i\beta z)$  and combine these equations according to \eqref{eq.og6kibc} and the Binet formulas \eqref{bine}.
\end{proof}
\begin{example} 	When $n=1$ and $z=1/2$, then from \eqref{th15_F}--\eqref{th16_L} we get the expressions
	\begin{align*}
	\sqrt5\sum_{k = 1}^\infty \frac{(-1)^{k-1}F_{2k}\zeta (2k)} {4^k k(2k+1)} &= \frac{\pi}{4}-2\ln\alpha+\frac{\pi\sinh(\ln\alpha)}{3} -\frac{\alpha\Li_{2}(e^{\pi\beta})+\beta\Li_{2}(e^{-\pi\alpha})}{\pi},
	\end{align*}
	\begin{align*}
	\sum_{k = 1}^\infty  \frac{(-1)^{k-1}L_{2k}\zeta (2k)} {4^k k(2k+1)}& =\frac{8+\pi\sqrt5}{4}-2\ln\pi - \frac{\pi\cosh(\ln\alpha)}{3}+\frac{\alpha\Li_{2}(e^{\pi\beta}) - \beta\Li_{2}(e^{-\pi\alpha})}{\pi},\\
	\sum_{k = 1}^\infty \frac{ F_{2k}\zeta (2k)} {4^k k(2k+1)} &= \frac{2\ln\alpha}{\sqrt5} -\frac{\beta\Cl_{2}(\pi\alpha)-\alpha\Cl_{2}(\pi\beta)}{\sqrt5\pi},\\
	\sum_{k = 1}^\infty \frac{L_{2k}\zeta (2k)} {4^k k(2k+1)} &= 2\ln\pi -2 -  \frac{ \beta\Cl_{2}(\pi\alpha) + \alpha\Cl_{2}(\pi\beta)}{\pi}.
	\end{align*}
\end{example}
\begin{theorem}
	If $n$ is a positive integer and $z$ is any real number such that  $0<|z|\leq {1}/{\alpha}$, then
	\begin{align*}
	\frac{\sqrt5}{2}	\sum_{k = 1}^\infty  \frac{F_{2k}\zeta(2k)z^{2k}}{k(k+n)} &= \frac{\ln\alpha}{n} -\frac{(-1)^n\sqrt5(2n-1)!F_{2n}\zeta(2n+1)}{(2\pi z)^{2n}}\\
	&\,\quad-(2n-1)!\sum_{j=1}^n (-1)^j\frac{\beta^{2j}\Cl_{2j+1}(2\pi\alpha z) -  \alpha^{2j}\Cl_{2j+1}(2\pi\beta z)}{(2n-2j)!(2\pi z)^{2j}}\\
	&\,\quad+(2n-1)!\sum_{j=1}^n (-1)^j \frac{\beta^{2j-1}\Cl_{2j}(2\pi\alpha z) - \alpha^{2j-1}\Cl_{2j}(2\pi\beta z)}{(2n+1-2j)!(2\pi z)^{2j-1}},\\
	\frac{1}{2}	\sum_{k = 1}^\infty  \frac{L_{2k}\zeta(2k)z^{2k}}{k(k+n)} &= -\frac{1}{2n^2}+\frac{\ln(2\pi z)}{n} -\frac{(-1)^n(2n-1)!L_{2n}\zeta(2n+1)}{(2\pi z)^{2n}}\\
	&\,\quad-(2n-1)!\sum_{j=1}^n\frac{(-1)^j\big(\beta^{2j}\Cl_{2j+1}(2\pi\alpha z) + \alpha^{2j}\Cl_{2j+1}(2\pi\beta z)\big)}{(2n-2j)!(2\pi z)^{2j}}\\
	&\,\quad+(2n-1)!\sum_{j=1}^n\frac{(-1)^j\big(\beta^{2j-1}\Cl_{2j}(2\pi\alpha z) + \alpha^{2j-1}\Cl_{2j}(2\pi\beta z)\big)}{(2n+1-2j)!(2\pi z)^{2j-1}}.
	\end{align*}
\end{theorem}
\begin{proof}
	Both results follow immediately from \eqref{eq.mag6ovt}. We omit details.
\end{proof}
\begin{theorem}
	If $n$ is a positive integer and $z$ is any real number such that  $0<|z|\leq{1}/{\alpha}$, then
	\begin{align*}
	\sqrt5	\sum_{k = 1}^\infty & \frac{F_{2k}\zeta(2k)z^{2k}}{k(2k-1+2n)} = \frac{1}{2n-1} \ln\left(-\alpha^2\,\frac{\sin(\pi\beta z)}{\sin(\pi\alpha z)}\right)\\
	&\qquad\qquad-(2n-2)!\sum_{j=1}^n (-1)^j\frac{(2n+1-2j)\Cl_{2j}(2\pi\alpha z) - 2\pi\alpha z \Cl_{2j-1}(2\pi\alpha z)}{(2n+1-2j)!(2\pi\alpha z)^{2j-1}}\\
	&\qquad\qquad+(2n-2)!\sum_{j=1}^n (-1)^j \frac{(2n+1-2j)\Cl_{2j}(2\pi\beta z) - 2\pi\beta z \Cl_{2j-1}(2\pi\beta z)}{(2n+1-2j)!(2\pi\beta z)^{2j-1}},
			\end{align*}
		\begin{align*}	
	\sum_{k = 1}^\infty & \frac{L_{2k}\zeta(2k)z^{2k}}{k(2k-1+2n)} = -\frac{2}{(2n-1)^2} + \frac{1}{2n-1} \ln\left(\frac{\pi^2z^2}{\sin(\pi\alpha z)\sin(\pi\beta z)}\right)\\
	&\qquad\qquad\quad\,\,-(2n-2)!\sum_{j=1}^n(-1)^j \frac{(2n+1-2j)\Cl_{2j}(2\pi\alpha z) - 2\pi\alpha z \Cl_{2j-1}(2\pi\alpha z)}{(2n+1-2j)!(2\pi\alpha z)^{2j-1}}\\
	&\qquad\qquad\quad\,\,-(2n-2)!\sum_{j=1}^n(-1)^j \frac{(2n+1-2j)\Cl_{2j}(2\pi\beta z) - 2\pi\beta z \Cl_{2j-1}(2\pi\beta z)}{(2n+1-2j)!(2\pi\beta z)^{2j-1}}.
	\end{align*}
\end{theorem}
\begin{proof}
	Both formulas follow immediately from \eqref{eq.no0fbhh}.
\end{proof}

\section{Two other interesting series}

Our analysis allows the evaluation of two other interesting series which are similar to the series considered by Orr \cite{orr}.
For positive integers $m$ and $n$ with $m\neq n$, let us consider the two series
\begin{equation*}
P(m,n,z) = \sum_{k = 1}^\infty \frac{(- 1)^k \zeta (2k) z^{2k}}{k(2k+m)(2k + n)}, \qquad 0<|z|\leq 1,
\end{equation*}
and its degenerated counterpart for $m\geq1$
\begin{equation*}
Q(m,z) = \sum_{k = 1}^\infty \frac{(- 1)^k \zeta (2k) z^{2k}}{k(2k+m)^2}, \qquad 0<z\leq 1,
\end{equation*}

Then, the following result holds true.
\begin{theorem}
	For non-equal positive integers $m$, $n$ and all $z$ such that $0<|z|\le 1$, we have
	\begin{align}
	\sum_{k = 1}^\infty & \frac{(- 1)^{k-1} \zeta (2k)z^{2k} }{k(2k+m)(2k + n)} = \frac{m+n}{(mn)^2} + \frac{{\pi z}}{(m+1)(n + 1)} - \frac{\ln(2\pi z)}{mn}\nonumber \\
	& \qquad\qquad\,\, + \frac{1}{m-n} \left ( \frac{(m - 1)!}{(2\pi z)^m } \zeta (m + 1) - \frac{(n - 1)!}{(2\pi z)^n } \zeta (n + 1) \right ) \nonumber\\
	& \qquad\qquad\,\, - \frac{1}{m-n}\left ((m - 1)!\sum_{j = 1}^{m} \frac{\Li_{j+1}\big( e^{- 2\pi z} \big)}{(m-j)!(2\pi z)^j} - (n - 1)!\sum_{j = 1}^{n} \frac{\Li_{j + 1} \big( e^{- 2\pi z} \big)}{(n-j)!(2\pi z)^j}\right)\label{main_mn_id1}
	\end{align}
	and, for all $0<z\leq 1$,
	\begin{align}
	\sum_{k = 1}^\infty \frac{(- 1)^{k-1} \zeta (2k) z^{2k}}{k(2k+m)^2} & = \frac{2}{m^3} + \frac{\pi z}{(m+1)^2} - \frac{\ln(2\pi z)}{m^2} \nonumber\\
	&\quad+ \frac{(m - 1)!}{(2\pi z)^m } \zeta (m + 1) H_{m-1} + \frac{m!}{m^2}\sum_{j = 1}^{m} \frac{\Li_{j+1}\left( e^{- 2\pi z}\right)}{(m-j)!(2\pi z)^j}\nonumber\\
	&\quad - \frac{(m-1)!}{(2\pi z)^m} \sum_{k = 1}^\infty \frac{\Ein(2\pi z k) + \sum_{j=1}^m  \frac{\Gamma(j,2\pi zk)}{j!}}{k^{m+1}},\label{main_msquared_id}
	\end{align}
	where $H_n=\sum\limits_{s=1}^n \frac{1}{s}$, $H_0=0$, are the harmonic numbers,
	$\Gamma(a,x)$ is the incomplete gamma function
	$
	\Gamma(a,x) = \int\limits_x^\infty t^{a-1}\,e^{-t}\,dt,
	$
	and
	$\Ein(x)$ being the exponential integral
	$
	\Ein(x) = \int\limits_0^x \frac{1-e^{-t}}{t}\,dt.
	$
\end{theorem}
\begin{proof}
	From the partial fraction decomposition
	\begin{equation*}
	\frac{1}{k(2k+m)(2k+n)} = \frac{2}{m(m-n)(2k+m)} - \frac{2}{n(m-n)(2k+n)} + \frac{1}{mnk}
	\end{equation*}
	we immediately see that we can write
	\begin{equation*}
	P(m,n,z) = \frac{1}{mn} S_1(z) + \frac{2}{m-n}\left ( \frac{1}{m} S_{2,m}(z) - \frac{1}{n} S_{2,n}(z) \right )\!,
	\end{equation*}
	where $S_1(z)$ is defined in \eqref{sum_s1} and $S_{2,v}(z)$ equals
	\begin{equation*}
	S_{2,v}(z) = \sum_{k=1}^\infty \frac{(-1)^k \zeta(2k) z^{2k}}{2k+v},
	\end{equation*}
	which was also evaluated in Section 2. To complete the proof of \eqref{main_mn_id1} we simplify making use of the elementary identity
	\begin{equation*}
	\frac{n(n+1)-m(m+1)}{(m+1)(n+1)(m-n)} + 1 = \frac{mn}{(m+1)(n+1)}.
	\end{equation*}

	For \eqref{main_msquared_id} we start with the partial fraction decomposition
	\begin{equation*}
	\frac{1}{k(2k+m)^2} = \frac{1}{m^2 k} - \frac{2}{m^2(2k+m)} - \frac{2}{m(2k+m)^2}.
	\end{equation*}
	This yields
	\begin{equation}\label{G_mz}
	Q(m,z) = \frac{1}{m^2} S_1(z) - \frac{2}{m^2} S_{2,m}(z) - \frac{2}{m} S_{3,m}(z)
	\end{equation}
	with $S_1(z)$ and $S_{2,m}(z)$ as above and
	\begin{equation}\label{S3}
	S_{3,m}(z) = \sum_{k=1}^\infty \frac{(-1)^k \zeta(2k) z^{2k}}{(2k+m)^2}.
	\end{equation}
	We have
	\begin{align*}
	S_{3,m}(z) & =  \int_0^1 y^{m-1} S_{2,m}(zy)\,dy \\
	& =  \frac{1}{2m^2} - \frac{\pi z}{2(m+1)^2} - \frac{m!}{2^{m+1}(\pi z)^m} \zeta(m+1) \int_0^1 \frac{dy}{y}   \\
	&  \quad + \pi z \sum_{j=0}^m \frac{m!}{j!} \frac{1}{(2\pi z)^{m+1-j}} \int_0^1 y^{j-1} \Li_{m+1-j}(e^{-2\pi zy})\, dy
	\end{align*}
	\begin{align*}
	& =  \frac{1}{2m^2} - \frac{\pi z}{2(m+1)^2} + \frac{m!}{2^{m+1}(\pi z)^m}
	\int_0^1 \frac{ \Li_{m+1}(e^{-2\pi zy}) - \zeta(m+1)}{y}  dy  \\
	&  \quad + \frac{\pi z}{(2\pi z)^{m+1}} \sum_{j=1}^m \frac{m!}{j!} (2\pi z)^{j} \int_0^1 y^{j-1} \Li_{m+1-j}(e^{-2\pi zy})\, dy.
	\end{align*}

	Next,
	\begin{align*}
	\int_0^1 \frac{1}{y} \left ( \Li_{m+1}(e^{-2\pi zy}) - \zeta(m+1)\right ) dy
	& =  \sum_{k=1}^\infty \frac{1}{k^{m+1}} \int_0^1 \frac{e^{-2\pi zyk} - 1 }{y} \, dy = - \sum_{k=1}^\infty \frac{\Ein(2\pi zk)}{k^{m+1}}.
	\end{align*}
	Also,
	\begin{align*}
	\int_0^1 y^{j-1} \Li_{m+1-j}(e^{-2\pi zy}) dy & =  \sum_{k=1}^\infty \frac{1}{k^{m+1-j}} \int_0^1 y^{j-1} e^{-2\pi zyk} dy \\
	& =  \sum_{k=1}^\infty \frac{1}{k^{m+1-j}} \frac{(j-1)! - \Gamma(j,2\pi zk)}{(2\pi zk)^j},
	\end{align*}
	as
	\begin{equation*}
	\int_0^1 x^{j-1} e^{-2\pi ax}\, dx = \frac{(j-1)! - \Gamma(j,2\pi a)}{(2\pi a)^j}, \qquad \Re (a)>0.
	\end{equation*}

	The expression for $S_{3,m}(z)$ becomes
	\begin{align*} 
	S_{3,m}(z) & =  \frac{1}{2m^2} - \frac{\pi z}{2(m+1)^2} + \frac{m!}{2^{m+1}(\pi z)^m}
	\left ( \sum_{k=1}^\infty  \frac{\gamma - \Ein(2\pi z k)}{k^{m+1}} - \gamma \zeta(m+1) \right ) \\
	&  \quad + \frac{\pi z}{(2\pi z)^{m+1}} \sum_{j=1}^m \frac{m!}{j!} (2\pi z)^{j} \sum_{k=1}^\infty \frac{1}{k^{m+1-j}} \frac{(j-1)! - \Gamma(j,2\pi zk)}{(2\pi zk)^j}\\
	& =  \frac{1}{2m^2} - \frac{\pi z}{2(m+1)^2} + \frac{m!}{2^{m+1}(\pi z)^m} \zeta(m+1) (H_m - \gamma) \\
	&  \quad + \frac{m!}{2^{m+1}(\pi z)^m} \sum_{k=1}^\infty  \frac{\gamma - \Ein(2\pi z k)}{k^{m+1}}
	- \frac{m!}{2^{m+1}(\pi z)^m} \sum_{j=1}^m \frac{1}{j!} \sum_{k=1}^\infty \frac{\Gamma(j,2\pi zk)}{k^{m+1}},
	\end{align*}
where $\gamma=\lim\limits_{n\to\infty}(H_n-\ln n)$ is the Euler--Mascheroni constant.

	The expression for $Q(m,z)$ follows from simplifications according to \eqref{G_mz}.
\end{proof}
\begin{example} Series identity \eqref{main_mn_id1} yields
	\begin{align*}
	\sum_{k = 1}^\infty \frac{(- 1)^{k-1} \zeta (2k)}{k(2k+1)(2k+2)}  &= \frac{3}{4} +\frac{\pi}{12} - \frac{\ln(2\pi)}{2} + \frac{\zeta(3)}{(2\pi)^2}-\frac{\Li_3(e^{-2\pi})}{(2\pi)^2},\\
	\sum_{k = 1}^\infty \frac{(- 1)^{k-1} \zeta (2k)}{4^k k(2k+1)(2k+2)} & = \frac{3}{4} -\frac{\pi}{12} - \frac{\ln\pi}{2} + \frac{\zeta(3)}{\pi^2}-\frac{\Li_3(e^{-\pi})}{\pi^2},\\
	\sum_{k = 1}^\infty \frac{(- 1)^{k-1} \zeta (2k)}{4^k k(2k+1)(2k+3)}  &= \frac{4}{9}-\frac{7\pi}{720} -\frac{\ln\pi}{3}  -\frac{\Li_3(e^{-\pi})}{\pi^2}-\frac{\Li_4(e^{-\pi})}{\pi^3}.
	\end{align*}
\end{example}
\begin{corollary}
	If $m$ and $n$ are non-equal positive integers having the same parity, then
	\begin{align*}
	& \sum_{k = 1}^\infty  {\frac{{\zeta (2k)}}{{k(2k + m)(2k + n)}}}  =  - \frac{{m + n}}{{(mn)^2 }} + \frac{{\ln (2\pi )}}{{mn}}\\
	&\qquad\qquad\quad - \frac{1 + ( - 1)^n}{{2(m - n)}}\left( {i^{m} (m - 1)!\frac{{\zeta (m + 1)}}{{(2\pi )^m }} - i^{n} (n - 1)!\frac{{\zeta (n + 1)}}{{(2\pi )^n }}} \right)\\
	&\qquad\qquad\quad + \frac{1}{{m - n}}\left( {(m - 1)!\sum_{j = 1}^{\left\lfloor {m/2} \right\rfloor } {\frac{{( - 1)^j \zeta (2j + 1)}}{{(m - 2j)!(2\pi )^{2j} }}}  - (n - 1)!\sum_{j = 1}^{\left\lfloor {n/2} \right\rfloor } {\frac{{( - 1)^j \zeta (2j + 1)}}{{(n - 2j)!(2\pi )^{2j} }}} } \right)\!.
	\end{align*}
\end{corollary}
\begin{proof}
	Evaluate \eqref{main_mn_id1} at $z=i$.
\end{proof}
\begin{example} We have
	\begin{align*}
	\sum_{k = 1}^\infty {\frac{{\zeta (2k)}}{{k(2k + 1)(2k + 3)}}} &= - \frac{4}{9} + \frac{1}{3}\ln (2\pi ) - \frac{{\zeta (3)}}{{4\pi ^2 }},\\
	\sum_{k = 1}^\infty {\frac{{\zeta (2k)}}{{k(k + 1)(k + 2)}}} &= - \frac{3}{8} + \frac{1}{2}\ln (2\pi ) - \frac{{3\zeta (3)}}{{2\pi ^2 }},\\
	\sum_{k = 1}^\infty {\frac{{\zeta (2k)}}{{k(2k + 3)(2k + 5)}}} & = - \frac{8}{{225}} + \frac{1}{{15}}\ln (2\pi )
	- \frac{{\zeta (3)}}{{4\pi ^2 }} + \frac{{3\zeta (5)}}{{4\pi ^4 }}.
	\end{align*}
\end{example}
\begin{corollary}
	If $m$ is a positive even integer and $n$ is a positive odd integer, then
	\begin{align*}
	&\sum_{k = 1}^\infty  {\frac{{\zeta (2k)}}{{k(2k + m)(2k + n)}}}  =  - \frac{{m + n}}{{(mn)^2 }} + \frac{{\ln (2\pi )}}{{mn}} -  \frac{i^m {(m - 1)!}}{{m - n}}\frac{{\zeta (m + 1)}}{{(2\pi )^m }}\\
	&\qquad\quad + \frac{1}{{m - n}}\left( {(m - 1)!\sum_{j = 1}^{m/2} {\frac{{( - 1)^j \zeta (2j + 1)}}{{(m - 2j)!(2\pi )^{2j} }}}  - (n - 1)!\sum_{j = 1}^{(n - 1)/2} {\frac{{( - 1)^j \zeta (2j + 1)}}{{(n - 2j)!(2\pi )^{2j} }}} } \right)\!.
	\end{align*}
\end{corollary}
\begin{example} We have
	\begin{align*}
	\sum_{k = 1}^\infty {\frac{{\zeta (2k)}}{{k(k + 2)(2k + 3)}}} & = - \frac{7}{{72}} + \frac{\ln (2\pi ) }{6}- \frac{{\zeta (3)}}{{2\pi ^2 }},\\
	\sum_{k = 1}^\infty {\frac{{\zeta (2k)}}{{k(k + 1)(2k + 5)}}} & =  - \frac{7}{{50}} + \frac{\ln (2\pi )}{5} - \frac{{2\zeta (3)}}{{3\pi ^2 }} + \frac{{\zeta (5)}}{{\pi ^4 }}.
	\end{align*}
\end{example}

\section{Concluding remarks}

We conclude with the following observations. Another natural generalization of the function $P(n,z)$ is one of the form
\begin{equation}\label{gen_P1}
P(n,z,p) = \sum_{k = 1}^\infty \frac{(- 1)^k \zeta (2k) z^{2k}}{k^p (2k + n)}, \qquad 0<|z|\leq 1,\,\, p\geq 1.
\end{equation}
The recursion
\begin{equation*}
\frac{1}{k^p (2k+n)} = \frac{1}{n k^p} - \frac{2}{n} \frac{1}{k^{p-1}(2k+n)}
\end{equation*}
can be solved by standard methods to get
\begin{equation*}
\frac{1}{k^p (2k+n)} = \sum_{j=0}^{p-1} \frac{(-2)^j}{n^{j+1} k^{p-j}} + \Big ( \frac{2}{n}\Big )^p \frac{(-1)^p}{2k+n}.
\end{equation*}
Then
\begin{equation}
P(n,z,p) = \sum_{j=0}^{p-1} \frac{(-2)^j }{n^{j+1}}\sum_{k = 1}^\infty \frac{(- 1)^k \zeta (2k) z^{2k}}{k^{p-j}}
+ \left (- \frac{2}{n}\right )^p S_{2,n}(z).
\end{equation}
Hence, to extend Theorem \ref{thm1} to $P(n,z,p)$ it suffices to find a closed form for
\begin{equation}\label{dilog_series}
\sum_{k=1}^\infty \frac{(-1)^k \zeta(2k) z^{2k}}{k^p} = \sum_{t=1}^\infty \Li_p\left (- \left (\frac{z}{t}\right )^2 \right), \quad p\geq 2,
\end{equation}
as
\begin{equation*}
\sum_{t=1}^\infty \Li_1\left (- \left (\frac{z}{t}\right )^2 \right) = - \ln\left (\frac{\sinh(\pi z)}{\pi z}\right ) = S_1 (z).
\end{equation*}

Similarly, from the partial fraction decomposition
\begin{equation*}
\frac{1}{k (2k+m)(2k+n)} = \frac{1}{m n k} + \frac{2}{m (m-n)(2k+m)} - \frac{2}{n (m-n)(2k+n)},
\end{equation*}
we also get
\begin{align*}
\frac{1}{k^p (2k+m)(2k+n)} & = \frac{1}{m n k^p} + \frac{2}{m (m-n)}\left ( \sum_{j=0}^{p-2} \frac{(-2)^j }{m^{j+1} k^{p-1-j}}
+ \Big ( \frac{2}{m}\Big )^{p-1} \frac{(-1)^{p-1}}{2k+m}\right ) \\
& \quad - \frac{2}{n (m-n)}\left ( \sum_{j=0}^{p-2} \frac{(-2)^j }{n^{j+1} k^{p-1-j}}
+ \Big ( \frac{2}{n}\Big )^{p-1} \frac{(-1)^{p-1}}{2k+n}\right ).
\end{align*}
This shows that a closed form evaluation of \eqref{dilog_series} would also allow us to evaluate the generalization of $P(m,n,z)$ to
\begin{equation*}\label{gen_P2}
P(m,n,z,p) = \sum_{k = 1}^\infty \frac{(- 1)^k \zeta (2k) z^{2k}}{k^p (2k+m)(2k + n)}, \qquad 0<|z|\leq 1, \,\,\, m\neq n, \,\,\,p\geq 1,
\end{equation*}
via
\begin{align*}
P(m,n,z,p) &= \frac{1}{m n}\sum_{k = 1}^\infty \frac{(- 1)^k \zeta (2k) z^{2k}}{k^{p}} \\
& \quad + \frac{2}{m (m-n)}\left ( \sum_{j=0}^{p-2} \frac{(-2)^j }{m^{j+1}}\sum_{k = 1}^\infty \frac{(- 1)^k \zeta (2k) z^{2k}}{k^{p-1-j}}
+ \left (- \frac{2}{m}\right )^{p-1} S_{2,m}(z)\right ) \\
& \quad - \frac{2}{n (m-n)}\left ( \sum_{j=0}^{p-2} \frac{(-2)^j }{n^{j+1}}\sum_{k = 1}^\infty \frac{(- 1)^k \zeta (2k) z^{2k}}{k^{p-1-j}}
+ \left (- \frac{2}{n}\right )^{p-1} S_{2,n}(z)\right ).
\end{align*}

Finally, considering the generalization of $Q(m,z)$ to $Q(m,z,p)$ defined by
\begin{equation*}\label{gen_Q}
Q(m,z,p) = \sum_{k = 1}^\infty \frac{(- 1)^k \zeta (2k) z^{2k}}{k^p (2k + m)^2}, \qquad 0<|z|\leq 1, \,\,p\geq 1,
\end{equation*}
we start with the partial fraction decomposition
\begin{equation*}
\frac{1}{k (2k+m)^2} = \frac{1}{m^2 k} - \frac{2}{m^2 (2k+m)} - \frac{2}{m (2k+m)^2}.
\end{equation*}
The recursion is solved as
\begin{equation*}
\frac{1}{k^p (2k+m)^2} = \sum_{j=0}^{p-1} \frac{(-2)^j }{m^{j+2} k^{p-j}}
- 2 \sum_{j=0}^{p-1} \frac{(-2)^j }{m^{j+2} k^{p-1-j} (2k+m)} + \left ( \frac{2}{m}\right )^p \frac{(-1)^p}{(2k+m)^2}
\end{equation*}
giving the identity
\begin{align*}
Q(m,z,p) & = \sum_{j=0}^{p-1} \frac{(-2)^j}{m^{j+2}} \sum_{k = 1}^\infty \frac{(- 1)^k \zeta (2k) z^{2k}}{k^{p-j}} \\
& \quad - 2 \sum_{j=0}^{p-1} \frac{(-2)^j }{m^{j+2}} P(m,z,p-1-j) + \left (- \frac{2}{m}\right )^{p} S_{3,m}(z),
\end{align*}
where $P(n,z,p)$ is defined in \eqref{gen_P1} and $S_{3,m}(z)$ is defined in \eqref{S3}.

\end{document}